\documentclass{amsart}
\usepackage{amsfonts}
\usepackage{amssymb}
\usepackage{amsmath}
\usepackage{amsthm}
\usepackage{amscd}
\usepackage{mathrsfs}
\usepackage{graphicx, color}
\usepackage{url}
\usepackage{tabularx}

\theoremstyle{plain}
\newtheorem{theorem}{Theorem}[section]

\theoremstyle{definition}

\newcommand{\As}[1]{
 {\mathrm{As}(#1)}
}

\newcommand{\Hom}[2]{
 {\mathrm{Hom}(#1, #2)}
}

\newcommand{\RotE}[1][]{
 {\mathrm{Rot}_{#1} \mathbb{E}^{2}}
}

\begin{document}

\title{Quandle homology and relative group homology}

\author{Ayumu Inoue}
\address{Department of Mathematics, Tsuda University, 2-1-1 Tsuda-machi, Kodaira-shi, Tokyo 187-8577, Japan}
\email{ayminoue@tsuda.ac.jp}


\begin{abstract}
We introduce a chain map from quandle homology to relative group homology, and construct several quandle cocycles through the chain map.
We also relate this chain map to triangulations of Seifert (hyper)surfaces of 1- and 2-dimensional links.
\end{abstract}

\maketitle

\section{Introduction}
\label{sec:introduction}

Quandles and their homology theories are closely related to knot theory.
Given a quandle, we can define a numerical invariant of an $n$-dimensional oriented link ($n = 1, 2$), called a coloring number.
Moreover, an $(n + 1)$-cocycle of the quandle gives rise to a refinement of the coloring number, called a cocycle invariant \cite{CJKLS2003}.
For example, Carter et al.\ \cite{CJKLS2003} used a cocycle invariant to show that the 2-twist-spun trefoil, a $2$-dimensional knot, is not invertible.
Satoh and Shima \cite{SS2004} proved, using a suitable cocycle invariant, that the triple point number of the 2-twist-spun trefoil is four.
Subsequently, several authors \cite{Ino2025, Sat2016, SS2005} applied similar methods to determine triple point numbers of specific 2-dimensional knots.
To the best of the author's knowledge, cocycle invariants currently provide the only known method for determining triple point numbers.

It is generally difficult to find explicit cocycles of a given quandle.
Therefore, it is useful to have methods for constructing cocycles of quandles.
For example, Kabaya and the author \cite{IK2014} introduced another homology theory of quandles, called simplicial quandle homology, together with a chain map from quandle homology to simplicial quandle homology.
Using this chain map, they constructed an explicit cocycle of a certain quandle and showed that the complex volume of a 1-dimensional link can be computed as (an element of) a cocycle invariant derived from the cocycle.
Subsequently, Kabaya \cite{Kab2012} established a method for constructing quandle cocycles from group cocycles via the chain map under a certain condition, and demonstrated that Mochizuki 3-cocycles of dihedral quandles \cite{Moc2003} can be reconstructed by this method.
Nosaka \cite{Nos2014} proposed another method for constructing quandle cocycles from group cocycles via the chain map under a different condition, and showed that Mochizuki 3-cocycles of Alexander quandles \cite{Moc2005} can also be reconstructed in this way.
From a completely different perspective, Bae, Carter and Kim \cite{BCK2021} also developed a method for constructing quandle 2-cocycles from group 2-cocycles.

In this paper, we introduce a chain map from quandle homology to relative group homology (Theorem \ref{thm:leveling}).
As an immediate consequence, this chain map provides a method for constructing quandle cocycles from relative group cocycles.
We apply this method to construct explicit cocycles for several quandles.
The chain map introduced by Kabaya and the author is closely related to triangulations of link complements (see \cite[Subsection 7.4]{IK2014}).
Similarly, we show that the chain map introduced in this paper is closely related to triangulations of Seifert (hyper)surfaces of links (Theorems \ref{thm:canonical_Seifert_surface}--\ref{thm:Seifert_hypersurface}).

This paper is organized as follows.
Sections \ref{sec:quandle_and_its_homology} and \ref{sec:relative_group_homology} respectively review quandle homology and relative group homology.
The chain map from quandle homology to relative group homology is established in Section \ref{sec:chain_map}.
We construct several quandle cocycles through the chain map in Section \ref{sec:quandle_cocycles}.
Finally, in Section \ref{sec:Seifert_surface}, we study a relationship between the chain map and triangulations of Seifert (hyper)surfaces.

\section{Quandle and its homology}
\label{sec:quandle_and_its_homology}

In this section, we briefly review the notions of quandles and their homologies.
We refer the reader to \cite{Kam2017} for details.

A \emph{quandle} is defined to be a non-empty set $X$ equipped with a binary operation $\ast : X \times X \to X$ satisfying the following three axioms.
\begin{itemize}
\item[(Q1)]
For each $x \in X$, $x \ast x = x$.
\item[(Q2)]
For each $x \in X$, the map $S_{x} : X \to X$ given by $S_{x}(w) = w \ast x$ is bijective.
\item[(Q3)]
For each $x, y, z \in X$, $(x \ast y) \ast z = (x \ast z) \ast (y \ast z)$.
\end{itemize}
By axioms (Q2) and (Q3), each $S_{x}$ is an automorphism of $X$.
In what follows, we mainly focus on the following quandle.
Let $G$ be a group, $\varphi$ an automorphism of $G$, and $H$ a subgroup of $G$ satisfying $\varphi(h) = h$ for all $h \in H$.
Then the set $H \backslash G$ of right cosets of $H$ in $G$ forms a quandle with a binary operation $\ast$ defined by
\[
 H g \ast H k = H \varphi(g k^{-1}) k.
\]
We let $(G, H, \varphi)$ denote the resulting quandle, and refer to it as a \emph{generalized Alexander quandle}.
If $H = \{ 1 \}$, then we abbreviate $(G, H, \varphi)$ as $(G, \varphi)$ via the identification $H \backslash G \cong G$.

Let $X$ be a quandle.
We define a group $\As{X}$ by the presentation
\[
 \As{X} = \langle \{ g_{x} \mid x \in X \} \mid \{ g_{x \ast y}^{-1} \> g_{y}^{-1} g_{x} g_{y} \mid x, y \in X \} \rangle.
\]
This group is called the \emph{associated group} of $X$.
We refer to a set equipped with a right action of $\As{X}$ as an \emph{$X$-set}.
A typical example of an $X$-set is $X$ itself, equipped with the right action of $\As{X}$ defined by $x \cdot g_{y} = x \ast y$ for $x, y \in X$.
In what follows, we mainly focus on the $X$-set $\mathbb{Z}$ endowed with the right action of $\As{X}$ given by $u \cdot g_{x} = u + 1$ for $u \in \mathbb{Z}$ and $x \in X$.

Fix an $X$-set $Y$.
Let $C^{R}_{m}(X)_{Y}$ be the free $\mathbb{Z}$-module generated by $Y \times X^{m}$ for $m \geq 1$, and $C^{R}_{0}(X)_{Y}$ the free $\mathbb{Z}$-module generated by $Y$.
Define a homomorphism $\partial_{m} : C^{R}_{m}(X)_{Y} \to C^{R}_{m - 1}(X)_{Y}$ by
\begin{align*}
 \partial_{m}(y; x_{1}, x_{2}, \dots, x_{m})
 & = \sum_{i = 1}^{m} (-1)^{i} (y; x_{1}, \dots, \widehat{x_{i}}, \dots, x_{m}) \\
 & \phantom{=} \ \, \enskip - \sum_{i = 1}^{m} (-1)^{i} (y \cdot g_{x_{i}}; S_{x_{i}}(x_{1}), \dots, S_{x_{i}}(x_{i - 1}), x_{i + 1}, \dots, x_{m}).
\end{align*}
Then $\partial_{m - 1} \circ \partial_{m} = 0$ for each $m \geq 2$ (see \cite{Kam2017}), and hence we obtain a chain complex $(C^{R}_{\ast}(X)_{Y}, \partial_{\ast})$.

For $m \geq 2$, let $C^{D}_{m}(X)_{Y}$ be the submodule of $C^{R}_{m}(X)_{Y}$ freely generated by the elements $(y; x_{1}, x_{2}, \dots, x_{m}) \in Y \times X^{m}$ satisfying $x_{i} = x_{i + 1}$ for some $i$ with $1 \leq i \leq m - 1$.
Set $C^{D}_{0}(X)_{Y} = C^{D}_{1}(X)_{Y} = 0$.
Since $(C^{D}_{\ast}(X)_{Y}, \partial_{\ast})$ is a subchain complex of $(C^{R}_{\ast}(X)_{Y}, \partial_{\ast})$ (see \cite{Kam2017}), the quotient modules
\[
 C^{Q}_{m}(X)_{Y} = C^{R}_{m}(X)_{Y} / C^{D}_{m}(X)_{Y}
\]
together with $\partial_{m}$ form a chain complex.
We let $H^{Q}_{m}(X)_{Y}$ denote the $m$-th homology group of this chain complex $(C^{Q}_{\ast}(X)_{Y}, \partial_{\ast})$, and refer to it as the \emph{$m$-th quandle homology group} of $(X, Y)$.
For an abelian group $A$, we let $H_{Q}^{m}(X; A)_{Y}$ denote the $m$-th cohomology group of the cochain complex
\[
 (C_{Q}^{\ast}(X; A)_{Y}, \delta^{\ast}) = (\Hom{C^{Q}_{\ast}(X)_{Y}}{A}, \Hom{\partial_{\ast + 1}}{A}),
\]
and refer to it as the \emph{$m$-th quandle cohomology group} of $(X, Y)$ with coefficients in $A$.
Here, $\Hom{C^{Q}_{m}(X)_{Y}}{A}$ denotes the $\mathbb{Z}$-module of homomorphisms $C^{Q}_{m}(X)_{Y} \to A$.

Let $O$ be a set consisting of a single element $o$, which is equipped with the trivial right action of $\As{X}$.
By convention, when $Y = O$, we drop the subscript $O$ from the above notations, and let $C^{R}_{0}(X) = 0$.
We also abbreviate $(o; x_{1}, x_{2}, \dots, x_{m})$ as $(x_{1}, x_{2}, \dots, x_{m})$.

\section{Relative group homology}
\label{sec:relative_group_homology}

In this section, we briefly review the notion of relative group homology in the sense of Adamson and Hochschild.
We refer the reader to \cite{AC2017} for details.

Let $G$ be a group, and $H$ a subgroup of $G$.
As usual, we let $G / H$ denote the set of left cosets of $H$ in $G$.
For each $m \geq 0$, $G$ acts on $(G / H)^{m + 1}$ from the left by
\[
 g \cdot (g_{0} H, g_{1} H, \dots, g_{m} H) = (g g_{0} H, g g_{1} H, \dots, g g_{m} H).
\]
We define an equivalence relation $\sim$ on $(G / H)^{m + 1}$ by
\[
 (g_{0} H, g_{1} H, \dots, g_{m} H) \sim g \cdot (g_{0} H, g_{1} H, \dots, g_{m} H)
\]
for some $g \in G$.
Let $C_{m}(G / H)$ be the free $\mathbb{Z}$-module generated by $(G / H)^{m + 1} / \sim$.
Considering each generator of $C_{m}(G / H)$ as an $m$-simplex whose vertices are labeled by elements of $G / H$, define a homomorphism $\partial_{m} : C_{m}(G / H) \to C_{m - 1}(G / H)$ by
\[
 \partial_{m}(g_{0} H, g_{1} H, \dots, g_{m} H)
 = \sum_{i = 0}^{m} (-1)^{i} (g_{0} H, \dots, \widehat{g_{i} H}, \dots, g_{m} H).
\]
Then $\partial_{m - 1} \circ \partial_{m} = 0$ for each $m \geq 2$ (see \cite{AC2017}), and hence we obtain a chain complex $(C_{\ast}(G / H), \partial_{\ast})$.
We let $H_{m}(G / H)$ denote the $m$-th homology group of $(C_{\ast}(G / H), \partial_{\ast})$, and refer to it as the \emph{$m$-th relative group homology group} of $(G, H)$.
For an abelian group $A$, we let $H^{m}(G / H; A)$ denote the $m$-th cohomology group of the cochain complex
\[
 (C^{\ast}(G / H; A), \delta^{\ast}) = (\Hom{C_{\ast}(G / H)}{A}, \Hom{\partial_{\ast + 1}}{A}),
\]
and refer to it as the \emph{$m$-th relative group cohomology group} of $(G, H)$ with coefficients in $A$.
By definition, if $H = \{ 1 \}$, then $H_{m}(G / H)$ (or $H^{m}(G / H; A)$) coincides with the $m$-th group homology (or cohomology) group of $G$ under the identification $G / H \cong G$.

We end this section with a remark.
For $m \geq 1$, let $D_{m}(G / H)$ be the submodule of $C_{m}(G / H)$ freely generated by the elements $(g_{0} H, g_{1} H, \dots, g_{m} H) \in (G / H)^{m + 1} / \sim$ satisfying $g_{i} H = g_{i + 1} H$ for some $i$ with $0 \leq i \leq m - 1$.
Set $D_{0}(G / H) = 0$.
Since $(D_{\ast}(G / H), \partial_{\ast})$ is a subchain complex of $(C_{\ast}(G / H), \partial_{\ast})$ (see \cite{Mac1963}), the quotient modules
\[
 \overline{C}_{m}(G / H) = C_{m}(G / H) / D_{m}(G / H)
\]
together with $\partial_{m}$ form a chain complex.
By Theorem 6.1 in Chapter VIII of \cite{Mac1963}, the $m$-th homology group of $(\overline{C}_{\ast}(G / H), \partial_{\ast})$ is isomorphic to that of $(C_{\ast}(G / H), \partial_{\ast})$.
Similarly, the $m$-th cohomology group of the cochain complex
\[
 (\overline{C}^{\, \ast}(G / H; A), \delta^{\ast}) = (\Hom{\overline{C}_{\ast}(G / H)}{A}, \Hom{\partial_{\ast + 1}}{A})
\]
is also isomorphic to that of $(C^{\ast}(G / H; A), \delta^{\ast})$.

\section{Chain map from quandle homology to relative group homology}
\label{sec:chain_map}

Let $G$ be a group, $\varphi$ an automorphism of $G$, and $H$ a subgroup of $G$ satisfying $\varphi(h) = h$ for all $h \in H$.
We let $X$ denote the generalized Alexander quandle $(G, H, \varphi)$.
The aim of this section is to introduce a chain map from $(C^{Q}_{\ast}(X)_{\mathbb{Z}}, \partial_{\ast})$ to $(\overline{C}_{\ast}(G / H), \partial_{\ast})$.
We begin by fixing notation and conventions.

For $m \geq 2$, let $I_{m}$ be the set of maps $\iota : \{ 2, 3, \dots, m \} \to \{ 0, 1 \}$.
For each $\iota \in I_{m}$, we let $|\iota|$ denote the number of $i \in \{ 2, 3, \dots, m \}$ satisfying $\iota(i) = 1$.
Given $\iota \in I_{m}$, $i \in \{ 1, 2, \dots, m \}$, and $(u; H g_{1}, H g_{2}, \dots, H g_{m}) \in C^{Q}_{m}(X)_{\mathbb{Z}}$, we define $H g(\iota, i) \in X$ by
\[
 H g(\iota, i)
 = S_{H g_{m}}^{\iota(m)} \circ \dots \circ S_{H g_{i + 2}}^{\iota(i + 2)} \circ S_{H g_{i + 1}}^{\iota(i + 1)} (H g_{i}).
\]
For $m \geq 0$ and $u \in \mathbb{Z}$, define an endomorphism $\varphi^{u}$ of $C_{m}(G / H)$ (or $\overline{C}_{m}(G / H)$) by
\[
 \varphi^{u}(g_{0} H, g_{1} H, \dots, g_{m} H) = (\varphi^{u}(g_{0}) H, \varphi^{u}(g_{1}) H, \dots, \varphi^{u}(g_{m}) H),
\]
where $\varphi^{u}$ on the right-hand side denotes the $u$-th power of the automorphism $\varphi$.

Fix $g \in G$.
Define a homomorphism $\psi : C^{R}_{m}(X)_{\mathbb{Z}} \to C_{m}(G / H)$ by
\begin{align*}
 & \psi(u; H g_{1}, H g_{2}, \dots, H g_{m}) \\
 & = \sum_{\iota \in I_{m}} (-1)^{|\iota| + 1} \varphi^{- |\iota| - u}(g^{-1} H, g(\iota, 1)^{-1} H, g(\iota, 2)^{-1} H, \dots, g(\iota, m)^{-1} H)
\end{align*}
for $m \geq 2$,
\[
 \psi(u; H g_{1})
 = - \varphi^{- u}(g^{-1} H, g_{1}^{-1} H)
\]
for $m = 1$, and
\[
 \psi(u)
 = - \varphi^{- u}(g^{-1} H)
 = - (H)
\]
for $m = 0$.

\begin{theorem}
\label{thm:leveling}
The homomorphism $\psi : C^{R}_{m}(X)_{\mathbb{Z}} \to C_{m}(G / H)$ is a chain map.
In particular, it induces the chain map $C^{Q}_{m}(X)_{\mathbb{Z}} \to \overline{C}_{m}(G / H)$.
\end{theorem}

\begin{proof}
We first prove that $\psi : C^{R}_{m}(X)_{\mathbb{Z}} \to C_{m}(G/H)$ is a chain map.
In the case $m = 1$, by a straightforward calculation, we obtain
\[
 \partial \circ \psi(u; H g_{1})
 = \psi \circ \partial(u; H g_{1})
 = 0.
\]
Suppose that $m \geq 2$.
For each $i \in \{ 2, 3, \dots, m \}$, let
\[
 I_{m}^{(i)} = \{ \iota \in I_{m} \mid \iota(i) = 0 \}.
\]
Given $\iota \in I_{m}^{(i)}$, define $\overline{\iota} \in I_{m}$ by
\[
 \overline{\iota}(j) =
 \begin{cases}
  1 & \text{if $j = i$}, \\
  \iota(j) & \text{otherwise}.
 \end{cases}
\]
We then have $|\overline{\iota}| = |\iota| + 1$.
By definition,
\begin{align*}
 & \partial \circ \psi(u; H g_{1}, H g_{2}, \dots, H g_{m}) \\
 & = \sum_{\iota \in I_{m}} (-1)^{|\iota| + 1} \varphi^{- |\iota| - u}(g(\iota, 1)^{-1} H, g(\iota, 2)^{-1} H, \dots, g(\iota, m)^{-1} H) \\
 & \enskip \quad + \sum_{\substack{1 \leq i \leq m, \\ \iota \in I_{m}}} (-1)^{|\iota| + i + 1} \varphi^{- |\iota| - u}(g^{-1} H, g(\iota, 1)^{-1} H, \dots, \widehat{g(\iota, i)^{-1} H}, \dots, g(\iota, m)^{-1} H).
\end{align*}
Assume that $\iota \in I_{m}^{(m)}$.
Then, for each $i \in \{ 1, 2, \dots, m - 1 \}$, we have
\[
 H g(\overline{\iota}, i)
 = H \varphi(g(\iota, i) g_{m}^{-1}) g_{m}
 = H \varphi(g(\iota, i)) \varphi(g_{m})^{-1} g_{m}.
\]
Moreover, since $H g(\iota, m) = H g_{m}$, we also have
\[
 H g(\overline{\iota}, m)
 = H g_{m}
 = H \varphi(g_{m}) \varphi(g_{m})^{-1} g_{m}
 = H \varphi(g(\iota, m)) \varphi(g_{m})^{-1} g_{m}.
\]
Therefore,
\begin{align*}
 & \sum_{\iota \in I_{m}} (-1)^{|\iota| + 1} \varphi^{- |\iota| - u}(g(\iota, 1)^{-1} H, g(\iota, 2)^{-1} H, \dots, g(\iota, m)^{-1} H) \\
 & = \sum_{\iota \in I_{m}^{(m)}} \Big( (-1)^{|\iota| + 1} \varphi^{- |\iota| - u}(g(\iota, 1)^{-1} H, g(\iota, 2)^{-1} H, \dots, g(\iota, m)^{-1} H) \\
 & \qquad + (-1)^{|\overline{\iota}| + 1} \varphi^{- |\overline{\iota}| - u}(g(\overline{\iota}, 1)^{-1} H, g(\overline{\iota}, 2)^{-1} H, \dots, g(\overline{\iota}, m)^{-1} H) \Big) \\
 & = \sum_{\iota \in I_{m}^{(m)}} \Big( (-1)^{|\iota| + 1} \varphi^{- |\iota| - u}(g(\iota, 1)^{-1} H, g(\iota, 2)^{-1} H, \dots, g(\iota, m)^{-1} H) \\
 & \qquad - (-1)^{|\iota| + 1} \varphi^{- |\iota| - u} (\varphi^{-1}(g_{m}^{-1}) g_{m} \cdot (g(\iota, 1)^{-1} H, g(\iota, 2)^{-1} H, \dots, g(\iota, m)^{-1} H)) \Big) \\
 & = 0.
\end{align*}
Hence, we have
\begin{align*}
 & \partial \circ \psi(u; H g_{1}, H g_{2}, \dots, H g_{m}) \\
 & = \sum_{\substack{1 \leq i \leq m, \\ \iota \in I_{m}}} (-1)^{|\iota| + i + 1} \varphi^{- |\iota| - u}(g^{-1} H, g(\iota, 1)^{-1} H, \dots, \widehat{g(\iota, i)^{-1} H}, \dots, g(\iota, m)^{-1} H).
\end{align*}
On the other hand,
\begin{align*}
 & \psi \circ \partial (u; H g_{1}, H g_{2}, \dots, H g_{m}) \\
 & = \psi \Big( \sum_{i = 1}^{m} (-1)^{i} (u; H g_{1}, \dots, \widehat{H g_{i}}, \dots, H g_{m}) \\
 & \qquad \qquad - \sum_{i = 1}^{m} (-1)^{i} (u + 1; S_{H g_{i}} (H g_{1}), \dots, S_{H g_{i}} (H g_{i - 1}), H g_{i + 1}, \dots, H g_{m}) \Big).
\end{align*}
For each $i$ with $2 \leq i \leq m$, we have
\begin{align*}
 & \psi(u; H g_{1}, \dots, \widehat{H g_{i}}, \dots, H g_{m}) \\
 & = \sum_{\iota \in I_{m}^{(i)}} (-1)^{|\iota| + 1} \varphi^{- |\iota| - u} (g^{-1} H, g(\iota, 1)^{-1} H, \dots, \widehat{g(\iota, i)^{-1} H}, \dots, g(\iota, m)^{-1} H).
\end{align*}
Moreover, since $S_{H g_{i}}$ is an automorphism of $X$, and $H g(\iota, j) = H g(\overline{\iota}, j)$ for each $\iota \in I_{m}^{(i)}$ and $j \in \{ i + 1, i + 2, \dots, m \}$, we have
\begin{align*}
 & \psi(u + 1; S_{H g_{i}} (H g_{1}), \dots, S_{H g_{i}} (H g_{i - 1}), H g_{i + 1}, \dots, H g_{m}) \\
 & = \sum_{\iota \in I_{m}^{(i)}} (-1)^{|\iota| + 1} \varphi^{- |\iota| - u - 1} (g^{-1} H, g(\overline{\iota}, 1)^{-1} H, \dots, \widehat{g(\overline{\iota}, i)^{-1} H}, \dots, g(\overline{\iota}, m)^{-1} H) \\
 & = - \sum_{\iota \in I_{m}^{(i)}} (-1)^{|\overline{\iota}| + 1} \varphi^{- |\overline{\iota}| - u} (g^{-1} H, g(\overline{\iota}, 1)^{-1} H, \dots, \widehat{g(\overline{\iota}, i)^{-1} H}, \dots, g(\overline{\iota}, m)^{-1} H).
\end{align*}
For $i = 1$, we have
\begin{align*}
 & \psi(u; \widehat{H g_{1}}, H g_{2}, \dots, H g_{m}) \\
 & = \sum_{\iota \in I_{m}^{(2)}} (-1)^{|\iota| + 1} \varphi^{- |\iota| - u} (g^{-1} H, \widehat{g(\iota, 1)^{-1} H}, g(\iota, 2)^{-1} H, \dots, g(\iota, m)^{-1} H).
\end{align*}
Moreover, since $H g(\iota, j) = H g(\overline{\iota}, j)$ for each $\iota \in I_{m}^{(2)}$ and $j \in \{ 2, 3, \dots, m \}$, we have
\begin{align*}
 & \psi(u + 1; H g_{2}, \dots, H g_{m}) \\
 & = \sum_{\iota \in I_{m}^{(2)}} (-1)^{|\iota| + 1} \varphi^{- |\iota| - u - 1} (g^{-1} H, \widehat{g(\iota, 1)^{-1} H}, g(\iota, 2)^{-1} H, \dots, g(\iota, m)^{-1} H) \\
 & = - \sum_{\iota \in I_{m}^{(2)}} (-1)^{|\overline{\iota}| + 1} \varphi^{- |\overline{\iota}| - u} (g^{-1} H, \widehat{g(\overline{\iota}, 1)^{-1} H}, g(\overline{\iota}, 2)^{-1} H, \dots, g(\overline{\iota}, m)^{-1} H).
\end{align*}
Therefore, we obtain
\begin{align*}
 & \psi \circ \partial (u; H g_{1}, H g_{2}, \dots, H g_{m}) \\
 & = \sum_{\iota \in I_{m}^{(2)}} (-1)^{|\iota| + 2} \varphi^{- |\iota| - u} (g^{-1} H, \widehat{g(\iota, 1)^{-1} H}, g(\iota, 2)^{-1} H, \dots, g(\iota, m)^{-1} H) \\
 & \quad + \sum_{\iota \in I_{m}^{(2)}} (-1)^{|\overline{\iota}| + 2} \varphi^{- |\overline{\iota}| - u} (g^{-1} H, \widehat{g(\overline{\iota}, 1)^{-1} H}, g(\overline{\iota}, 2)^{-1} H, \dots, g(\overline{\iota}, m)^{-1} H) \\
 & \quad + \sum_{\substack{2 \leq i \leq m, \\ \iota \in I_{m}^{(i)}}} (-1)^{|\iota| + i + 1} \varphi^{- |\iota| - u} (g^{-1} H, g(\iota, 1)^{-1} H, \dots, \widehat{g(\iota, i)^{-1} H}, \dots, g(\iota, m)^{-1} H) \\
 & \quad + \sum_{\substack{2 \leq i \leq m, \\ \iota \in I_{m}^{(i)}}} (-1)^{|\overline{\iota}| + i + 1} \varphi^{- |\overline{\iota}| - u} (g^{-1} H, g(\overline{\iota}, 1)^{-1} H, \dots, \widehat{g(\overline{\iota}, i)^{-1} H}, \dots, g(\overline{\iota}, m)^{-1} H) \\
 & = \sum_{\substack{1 \leq i \leq m, \\ \iota \in I_{m}}} (-1)^{|\iota| + i + 1} \varphi^{- |\iota| - u}(g^{-1} H, g(\iota, 1)^{-1} H, \dots, \widehat{g(\iota, i)^{-1} H}, \dots, g(\iota, m)^{-1} H) \\
& = \partial \circ \psi(u; H g_{1}, H g_{2}, \dots, H g_{m}).
\end{align*}

We next prove that $\psi$ induces the chain map $C^{Q}_{m}(X)_{\mathbb{Z}} \to \overline{C}_{m}(G / H)$.
It suffices to show that $\psi$ maps $C^{D}_{m}(X)_{\mathbb{Z}}$ into $D_{m}(G / H)$.
In the cases $m = 0, 1$, this claim follows from $C^{D}_{0}(X)_{\mathbb{Z}} = C^{D}_{1}(X)_{\mathbb{Z}} = 0$.
For $m \geq 2$, suppose that $(u; H g_{1}, H g_{2}, \dots, H g_{m}) \in C^{D}_{m}(X)_{\mathbb{Z}}$, i.e., $H g_{i} = H g_{i + 1}$ for some $i$ with $1 \leq i \leq m - 1$.
Then, since $S_{H g_{i + 1}} (H g_{i}) = H g_{i}$ by the axiom (Q1), we have $H g(\iota, i) = H g(\iota, i + 1)$ for each $\iota \in I_{m}$.
Therefore, we obtain $\psi(u; H g_{1}, H g_{2}, \dots, H g_{m}) \in D_{m}(G / H)$.
\end{proof}

\begin{theorem}
\label{thm:uniqueness}
The induced homomorphism $\psi : H^{Q}_{m}(X)_{\mathbb{Z}} \to H_{m}(G / H)$ does not depend on the choice of $g \in G$.
\end{theorem}

\begin{proof}
Let $\psi$ and $\psi^{\prime}$ be chain maps $C^{Q}_{m}(X)_{\mathbb{Z}} \to \overline{C}_{m}(G/H)$ determined by elements $g$ and $g^{\prime}$ of $G$, respectively.
Define a homomorphism $\Psi : C^{Q}_{m}(X)_{\mathbb{Z}} \to \overline{C}_{m + 1}(G / H)$ by
\begin{align*}
 & \Psi(u; H g_{1}, H g_{2}, \dots, H g_{m}) \\
 & = \sum_{\iota \in I_{m}} (-1)^{|\iota| + 1} \varphi^{- |\iota| - u}((g^{\prime})^{-1} H, g^{-1} H, g(\iota, 1)^{-1} H, g(\iota, 2)^{-1} H, \dots, g(\iota, m)^{-1} H)
\end{align*}
for $m \geq 2$,
\[
 \Psi(u; H g_{1})
 = - \varphi^{- u}((g^{\prime})^{-1} H, g^{-1} H, g_{1}^{-1} H)
\]
for $m = 1$, and
\[
 \Psi(u)
 = - \varphi^{- u}((g^{\prime})^{-1} H, g^{-1} H)
\]
for $m = 0$.
Then a similar argument to the proof of Theorem \ref{thm:leveling} shows that $\partial \circ \Psi + \Psi \circ \partial = \psi - \psi^{\prime}$, i.e., $\Psi$ is a chain homotopy between $\psi$ and $\psi^{\prime}$.
\end{proof}

By Theorem \ref{thm:leveling}, for an abelian group $A$, each relative group $m$-cocycle $f \in \overline{C}^{\, m}(G / H; A)$ yields a quandle $m$-cocycle $f \circ \psi \in C_{Q}^{m}(X; A)_{\mathbb{Z}}$.
The cohomology class of $f \circ \psi$ in $H_{Q}^{m}(X; A)_{\mathbb{Z}}$ is independent of the choice of $g \in G$ by Theorem \ref{thm:uniqueness}.

\section{Quandle cocycles via the chain map}
\label{sec:quandle_cocycles}

In this section, we construct several concrete quandle cocycles via the chain map.
Since we only consider the case where $H$ is trivial, we abbreviate $\overline{C}^{\, m}(G / H; A)$ as $\overline{C}^{\, m}(G; A)$ under the identification $G / H \cong G$.

\subsection{}
\label{subsec:dihedral_quandle}

Let $G$ be the additive cyclic group of order $q \geq 3$, and $\varphi$ the involution of $G$ (i.e., $\varphi(x) = - x$).
We let $R_{q}$ denote the generalized Alexander quandle $(G, \varphi)$ (it is called the dihedral quandle of order $q$).
By definition, we have $x \ast y = 2 y - x$ for each $x, y \in R_{q}$.

For each integer $x$, let $\overline{x}$ denote the integer satisfying $\overline{x} \equiv x \pmod{q}$ and $0 \leq \overline{x} \leq q - 1$.
Define $b_{1} \in \overline{C}^{\, 1}(G; \mathbb{Z} / q \mathbb{Z})$ and $b_{2} \in \overline{C}^{\, 2}(G; \mathbb{Z} / q \mathbb{Z})$ respectively by
\begin{align*}
 b_{1}(0, x) & = x, &
 b_{2}(0, x, y) & = \frac{\overline{y - x} - \overline{y} + \overline{x}}{q} =
 \begin{cases}
  1 & \text{if $\overline{x} > \overline{y}$}, \\
  0 & \text{otherwise}.
 \end{cases}
\end{align*}
Then, as mentioned in \cite[Subsection 7A]{Kab2012}, $b_{1}$ and $b_{2}$ are non-trivial cocycles, and their cup products are (possibly trivial) cocycles.\footnote{We note that $b_{1}$ and $b_{2}$ are defined in \cite{Kab2012} using inhomogeneous (or bar) notation, whereas we use homogeneous notation throughout this paper.}
Let $b_{3} \in \overline{C}^{\, 3}(G; \mathbb{Z} / q \mathbb{Z})$ be the cup product of $b_{1}$ and $b_{2}$:
\begin{align*}
 b_{3}(0, x, y, z)
 & = (b_{1} \smile b_{2}) (0, x, y, z) \\
 & = b_{1} (0, x) \cdot b_{2} (x, y, z) \\
 & = b_{1} (0, x) \cdot b_{2} (0, y - x, z - x) \\
 & = x \cdot \frac{\overline{z - y} - \overline{z - x} + \overline{y - x}}{q} \\
 & =
 \begin{cases}
  x & \text{if ($\overline{x} > \overline{y} > \overline{z}$) or ($\overline{x} \neq \overline{y}$ and $\overline{x} = \overline{z}$)}, \\
  0 & \text{otherwise}.
 \end{cases}
\end{align*}
We note that $b_{3}$ is a non-trivial cocycle.
Indeed, we have
\begin{align*}
 \partial((0, 2, 1, 2) + (0, q - 2, q - 1, 0) + (0, q - 2, 0, q - 1)) & = 0, \\
 b_{3}((0, 2, 1, 2) + (0, q - 2, q - 1, 0) + (0, q - 2, 0, q - 1)) & = 2.
\end{align*}

For $m = 1, 2, 3$, we obtain quandle $m$-cocycles $\zeta_{m} \in C_{Q}^{m}(R_{q}; \mathbb{Z} / q \mathbb{Z})_{\mathbb{Z}}$ from $b_{m}$ via the chain map (setting $g = 0$):
\begin{align*}
 \zeta_{1}(u; x)
 & = - \> b_{1}(0, (-1)^{u + 1} x)
 = (-1)^{u} x, \\
 \zeta_{2}(u; x, y)
 & = - \> b_{2}(0, (-1)^{u + 1} x, (-1)^{u + 1} y) + b_{2}(0, (-1)^{u} (2 y - x), (-1)^{u} y) \\
 & =
 \begin{cases}
  1 & \text{if ($x = 0$, $y \neq 0$, and $(-1)^{u} 2 \, \overline{y} < (-1)^{u} q$) or} \\
  & \text{\phantom{if} ($x \neq 0, 2y$, $y \neq 0$, and $(-1)^{u} \, \overline{y} < (-1)^{u} \, \overline{x}, (-1)^{u} \, \overline{2 y - x}$)}, \\
  -1 & \text{if ($x = 2 y$, $y \neq 0$, and $(-1)^{u} 2 \, \overline{y} > (-1)^{u} q$) or} \\
  & \text{\phantom{if} ($x \neq 0, 2y$, $y \neq 0$, and $(-1)^{u} \, \overline{y} > (-1)^{u} \, \overline{x}, (-1)^{u} \, \overline{2 y - x}$)}, \\
  0 & \text{otherwise},
 \end{cases} \\
 \zeta_{3}(u; x, y, z)
 & = - \> b_{3}(0, (-1)^{u + 1} x, (-1)^{u + 1} y, (-1)^{u + 1} z) \\
 & \qquad + b_{3}(0, (-1)^{u} (2 y - x), (-1)^{u} y, (-1)^{u} z) \\
 & \qquad \quad + b_{3}(0, (-1)^{u} (2 z - x), (-1)^{u} (2 z - y), (-1)^{u} z) \\
 & \qquad \qquad - b_{3}(0, (-1)^{u + 1} (2 z - 2 y + x), (-1)^{u + 1} (2 z - y), (-1)^{u + 1} z) \\
 & =
 \begin{cases}
  (-1)^{u} 4 z & \text{if $[(-1)^{u} \, \overline{x} < (-1)^{u} \, \overline{y} < (-1)^{u} \, \overline{2 y - x} \leq (-1)^{u} \, \overline{z} \leq]$}, \\
  (-1)^{u} 2 z & \text{if $[(-1)^{u} \, \overline{x} < (-1)^{u} \, \overline{2 y - x} < (-1)^{u} \, \overline{y} < (-1)^{u} \, \overline{z} \leq]$}, \\
  & \text{\phantom{if} $[(-1)^{u} \, \overline{x} < (-1)^{u} \, \overline{2 y - x} \leq (-1)^{u} \, \overline{z} < (-1)^{u} \, \overline{y} <]$}, \\
  & \text{\phantom{if} $[(-1)^{u} \, \overline{x} < (-1)^{u} \, \overline{y} < (-1)^{u} \, \overline{z} < (-1)^{u} \, \overline{2 y - x} <]$}, \\
  & \text{\phantom{if} $[(-1)^{u} \, \overline{x} < (-1)^{u} \, \overline{z} < (-1)^{u} \, \overline{y} < (-1)^{u} \, \overline{2 y - x} <]$, or} \\
  & \text{\phantom{if} ($\overline{x} \neq \overline{y}$, $\overline{x} \neq \overline{z}$, $\overline{y} \neq \overline{z}$, and $\overline{x} = \overline{2 y - x}$)}, \\
  0 & \text{otherwise}.
 \end{cases}
\end{align*}
Here, $[x_{1} \diamond_{1} x_{2} \diamond_{2} x_{3} \diamond_{3} x_{4} \> \diamond_{4}]$ means that one of the following inequalities holds:
\begin{align*}
 & x_{1} \diamond_{1} x_{2} \diamond_{2} x_{3} \diamond_{3} x_{4}, &
 & x_{2} \diamond_{2} x_{3} \diamond_{3} x_{4} \diamond_{4} x_{1}, \\
 & x_{3} \diamond_{3} x_{4} \diamond_{4} x_{1} \diamond_{1} x_{2}, &
 & x_{4} \diamond_{4} x_{1} \diamond_{1} x_{2} \diamond_{2} x_{3}.
\end{align*}
Since the above values depend on $u$ only through its parity, each $\zeta_{m}$ also defines a quandle cocycle in $C_{Q}^{m}(R_{q}; \mathbb{Z} / q \mathbb{Z})_{\mathbb{Z} / 2 \mathbb{Z}}$, where $\As{R_{q}}$ acts on $\mathbb{Z} / 2 \mathbb{Z}$ from the right by $u \cdot g_{x} = u + 1$ for $u \in \mathbb{Z} / 2 \mathbb{Z}$ and $x \in R_{q}$.

\subsection{}
\label{subsec:rotE2_quandle}

For each $\theta \in \mathbb{R}$ with $0 \leq \theta < 2 \pi$, define an automorphism $\varphi_{\theta}$ of $\mathbb{C}$ by $\varphi_{\theta}(z) = z e^{\theta \sqrt{-1}}$.
Here, we consider $\mathbb{C}$ as a discrete additive group.
We let $\RotE[\theta]$ denote the generalized Alexander quandle $(\mathbb{C}, \varphi_{\theta})$.
By definition, we have
\[
 z \ast w = (z - w) e^{\theta \sqrt{-1}} + w
\]
for each $z, w \in \RotE[\theta]$.
We note that $z \ast w$ is the image of $z$ under the $\theta$-rotation about $w$.

Define a map $s : \mathbb{C}^{3} \to \mathbb{R}$ by
\[
 s(z_{0}, z_{1}, z_{2}) = \varepsilon \cdot (\text{the area of the triangle $(z_{0}, z_{1}, z_{2})$}),
\]
where $\varepsilon = + 1$ if $z_{0}, z_{1}, z_{2}$ are arranged counterclockwise, $\varepsilon = - 1$ if they are arranged clockwise, and the area of a degenerate triangle is zero.
Since $s(z_{0}, z_{1}, z_{2}) = 0$ if $z_{0} = z_{1}$ or $z_{1} = z_{2}$, and $s(z_{0} + w, z_{1} + w, z_{2} + w) = s(z_{0}, z_{1}, z_{2})$ for each $w \in \mathbb{C}$, $s$ induces a homomorphism $\overline{C}_{2}(\mathbb{C}) \to \mathbb{R}$, which we also denote by $s$.
Given $z_{0}, z_{1}, z_{2}, z_{3} \in \mathbb{C}$, consider the ``degenerate'' tetrahedron $(z_{0}, z_{1}, z_{2}, z_{3})$ in $\mathbb{C}$.
Since the boundary of this tetrahedron consists of $+ (z_{1}, z_{2}, z_{3})$, $- (z_{0}, z_{2}, z_{3})$, $+ (z_{0}, z_{1}, z_{3})$, and $- (z_{0}, z_{1}, z_{2})$, we obtain
\[
 s(\partial(z_{0}, z_{1}, z_{2}, z_{3}))
 = s(z_{1}, z_{2}, z_{3}) - s(z_{0}, z_{2}, z_{3}) + s(z_{0}, z_{1}, z_{3}) - s(z_{0}, z_{1}, z_{2})
 = 0.
\]
Hence, $s$ is a 2-cocycle (see \cite[Proof of Lemma 5.4]{Sat2026} for an alternative proof).

We obtain a quandle 2-cocycle $\xi \in C_{Q}^{2}(\RotE[\theta]; \mathbb{R})_{\mathbb{Z}}$ from $s$ via the chain map (setting $g = 0$):
\begin{align*}
 \xi(u; z_{1}, z_{2})
 & = - \> s(0, - z_{1} e^{- u \theta \sqrt{-1}}, - z_{2} e^{- u \theta \sqrt{-1}}) \\
 & \qquad + s(0, - (z_{1} \ast z_{2}) e^{- (u + 1) \theta \sqrt{-1}}, - z_{2} e^{- (u + 1) \theta \sqrt{-1}}) \\
 & = - \> s(0, - z_{1}, - z_{2}) + s(0, - (z_{1} \ast z_{2}), - z_{2}) \\
 & = - \> s(0, z_{1}, z_{2}) + s(0, z_{1} \ast z_{2}, z_{2}).
\end{align*}
Here, the second and third equalities follow from
\[
 s(w_{0} e^{\eta \sqrt{-1}}, w_{1} e^{\eta \sqrt{-1}}, w_{2} e^{\eta \sqrt{-1}}) = s(w_{0}, w_{1}, w_{2})
\]
for each $w_{i} \in \mathbb{C}$ and $\eta \in \mathbb{R}$.
Since the value of $\xi(u; z_{1}, z_{2})$ is independent of $u$, $\xi$ also defines a quandle 2-cocycle in $C_{Q}^{2}(\RotE[\theta]; \mathbb{R})$:
\[
 \xi(z_{1}, z_{2})
 = - \> s(0, z_{1}, z_{2}) + s(0, z_{1} \ast z_{2}, z_{2}).
\]
Let
\[
 \RotE = \bigsqcup_{0 \leq \theta < 2 \pi} \RotE[\theta].
\]
Then $\RotE$ forms a quandle with a binary operation $\ast$ defined as follows:
for $z \in \RotE[\theta]$ and $w \in \RotE[\eta]$,
\[
 z \ast w = (z - w) e^{\eta \sqrt{-1}} + w.
\]
Since the value of $\xi(z_{1}, z_{2})$ is independent of $\theta$, $\xi$ also defines a quandle 2-cocycle in $C_{Q}^{2}(\RotE; \mathbb{R})$.
This $\xi$ is applied in \cite{Sat2026} to the study of 1-dimensional knots.

\section{The chain map and Seifert (hyper)surface}
\label{sec:Seifert_surface}

We close this paper to study a relationship between the chain map established in Section \ref{sec:chain_map} and triangulations of Seifert (hyper)surfaces of 1- and 2-dimensional links.
The author hopes that this study will contribute to further developments in knot theory.

\subsection{}
\label{subsec:Seifert_surface_preliminaries}

We begin by briefly reviewing how quandles and their homology theories are related to knot theory.
We refer the reader to \cite{Kam2017} for details.

Let $X$ be a quandle, $Y$ an $X$-set, $L$ an oriented $n$-dimensional link ($n = 1, 2$), and $D$ a diagram of $L$.
A map $\mathscr{A} : \{ \text{arcs (or sheets) of $D$} \} \to X$ is said to be an \emph{$X$-coloring} of $D$ if
\[
 \mathscr{A}(a) \ast \mathscr{A}(b) = \mathscr{A}(c)
\]
holds at each crossing (or double point) of $D$, where $a$, $b$ and $c$ denote the arcs (or sheets) around the crossing (or double point) as depicted in Figure \ref{fig:around_crossing_point_and_regular_point} (A).
Moreover, a pair consisting of an $X$-coloring $\mathscr{A}$ of $D$ and a map $\mathscr{R} : \{ \text{regions of $D$} \} \to Y$ is said to be an \emph{$(X, Y)$-coloring} of $D$ if
\[
 \mathscr{R}(r) \cdot g_{\mathscr{A}(a)} = \mathscr{R}(s)
\]
holds around each arc (or sheet) $a$ of $D$, where $r$ and $s$ denote the regions around $a$ as depicted in Figure \ref{fig:around_crossing_point_and_regular_point} (B).
\begin{figure}[htbp]
 \centering
 \includegraphics[scale=0.25]{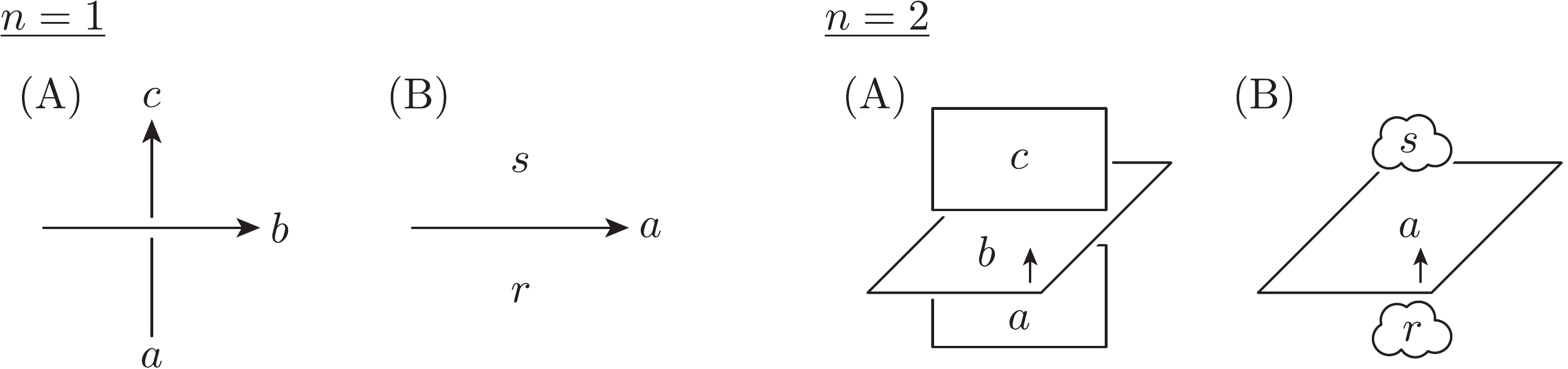}
 \caption{(A) The arcs (or sheets) around a crossing (or a double point). (B) The regions around an arc (or a sheet) $a$.}
 \label{fig:around_crossing_point_and_regular_point}
\end{figure}

Given an $(X, Y)$-coloring $(\mathscr{A}, \mathscr{R})$ of $D$, define $C(\mathscr{A}, \mathscr{R}) \in C^{R}_{n + 1}(X)_{Y}$ to be the sum of
\[
 \varepsilon (\mathscr{R}(r); \mathscr{A}(a), \mathscr{A}(b)) \quad (\text{or} \enskip \varepsilon (\mathscr{R}(r); \mathscr{A}(a), \mathscr{A}(b), \mathscr{A}(c)))
\]
over all crossings (or triple points) of $D$, where $\varepsilon$, $r$, $a$, $b$ (and $c$) are as depicted in Figure \ref{fig:around_crossing_point_and_triple_point}.
We can verify that $\partial C(\mathscr{A}, \mathscr{R}) = 0$, i.e., $C(\mathscr{A}, \mathscr{R})$ defines an $(n + 1)$-cycle in both $C^{R}_{n + 1}(X)_{Y}$ and $C^{Q}_{n + 1}(X)_{Y}$ (see \cite{CKS2001} for related arguments).
\begin{figure}[htbp]
 \centering
 \includegraphics[scale=0.25]{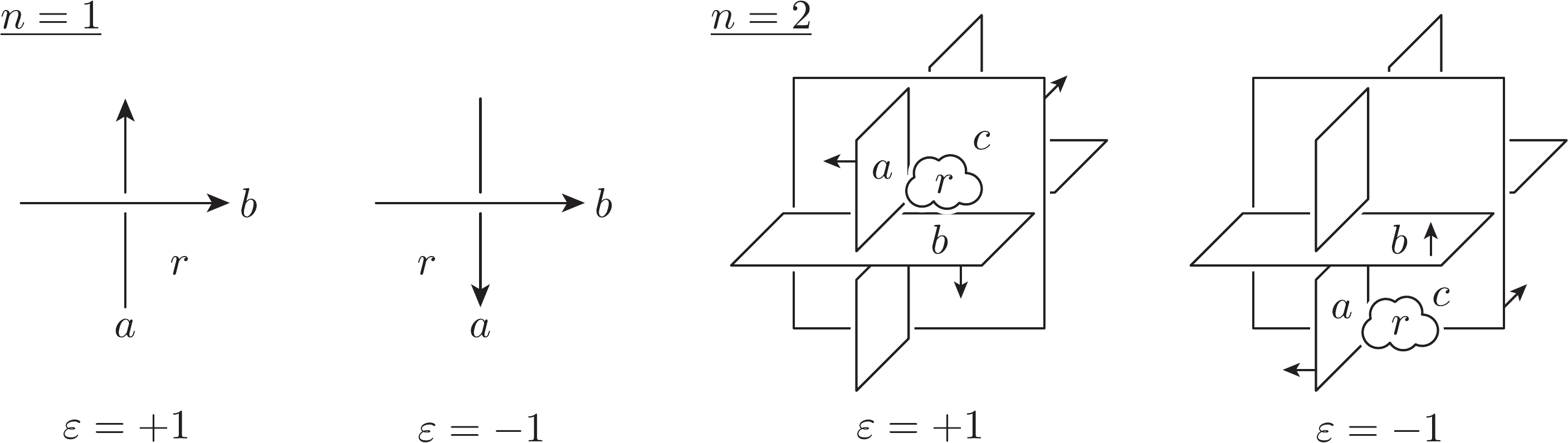}
 \caption{Arcs (or sheets) and a region around a crossing (or a triple point), and the value of $\varepsilon$.}
 \label{fig:around_crossing_point_and_triple_point}
\end{figure}

A Reidemeister (or Roseman) move naturally induces a bijection between the $(X, Y)$-colorings of the original diagram and those of the resulting diagram (see \cite{Kam2017}).
Therefore, the number of $(X, Y)$-colorings defines an invariant of $L$, called the $(X, Y)$-coloring number.
Moreover, this bijection induces a bijection between the multi-sets consisting of quandle $(n + 1)$-cycles $C(\mathscr{A}, \mathscr{R})$ derived from $(X, Y)$-colorings of the original diagram and those of the resulting diagram.
Hence, given a quandle $(n + 1)$-cocycle $\theta$ in $C_{Q}^{n + 1}(X; A)_{Y}$, the multi-set consisting of values of $\theta$ on $C(\mathscr{A}, \mathscr{R})$ defines an invariant of $L$, called the cocycle invariant.

\subsection{}
\label{subsec:Seifert_surface_1}

Let $L$ be an oriented 1-dimensional link, and $D$ a diagram of $L$.
Suppose that $D$ is connected and has at least one crossing.
Let $S_{D}$ be a canonical Seifert surface of $L$ derived from $D$ via the Seifert algorithm.
We note that $S_{D}$ is connected by the assumption that $D$ is connected.
We let $\overline{S_{D}}$ denote the closed surface obtained from $S_{D}$ by filling in its boundaries with disks.

Let $X$ be a generalized Alexander quandle $(G, H, \varphi)$.

\begin{theorem}
\label{thm:canonical_Seifert_surface}
For each $(X, \mathbb{Z})$-coloring $(\mathscr{A}, \mathscr{R})$ of $D$, the relative group 2-cycle $\psi(C(\mathscr{A}, \mathscr{R})) \in C_{2}(G / H)$ represents a triangulation of $\overline{S_{D}}$.
\end{theorem}

\begin{proof}
By definition, $\psi(C(\mathscr{A}, \mathscr{R}))$ is the sum of
\begin{align*}
 & \psi(\varepsilon (\mathscr{R}(r); \mathscr{A}(a), \mathscr{A}(b))) \\
 & = \psi(\varepsilon (u; H g_{1}, H g_{2})) \\
 & = - \> \varepsilon (\varphi^{-u}(g)^{-1} H, \varphi^{-u}(g_{1})^{-1} H, \varphi^{-u}(g_{2})^{-1} H) \\
 & \qquad + \varepsilon (\varphi^{- u - 1}(g)^{-1} H, \varphi^{- u - 1}(\varphi(g_{1} g_{2}^{-1}) g_{2})^{-1} H, \varphi^{- u - 1}(g_{2})^{-1} H)
\end{align*}
over all crossings of $D$, where $\varepsilon$, $r$, $a$ and $b$ are as depicted in Figure \ref{fig:around_crossing_point_and_triple_point}.
Since
\begin{align*}
 \varphi^{- u - 1}(\varphi(g_{1} g_{2}^{-1}) g_{2})^{-1}
 & = (\varphi^{- u - 1}(g_{2})^{-1} \varphi^{-u}(g_{2})) \cdot \varphi^{-u}(g_{1})^{-1}, \\
 \varphi^{- u - 1}(g_{2})^{-1}
 & = (\varphi^{- u - 1}(g_{2})^{-1} \varphi^{-u}(g_{2})) \cdot \varphi^{-u}(g_{2})^{-1},
\end{align*}
we can glue triangles
\begin{align*}
 T_{1} & = - \> \varepsilon (\varphi^{-u}(g)^{-1} H, \varphi^{-u}(g_{1})^{-1} H, \varphi^{-u}(g_{2})^{-1} H), \\
 T_{2} & = + \> \varepsilon (\varphi^{- u - 1}(g)^{-1} H, \varphi^{- u - 1}(\varphi(g_{1} g_{2}^{-1}) g_{2})^{-1} H, \varphi^{- u - 1}(g_{2})^{-1} H)
\end{align*}
along the edge
\begin{align*}
 E
 & = (\varphi^{-u}(g_{1})^{-1} H, \varphi^{-u}(g_{2})^{-1} H) \\
 & = (\varphi^{- u - 1}(g_{2})^{-1} \varphi^{-u}(g_{2})) \cdot (\varphi^{-u}(g_{1})^{-1} H, \varphi^{-u}(g_{2})^{-1} H).
\end{align*}
As illustrated in Figure \ref{fig:gluing_at_crossing}, we may perform this gluing at an appropriate position in the ambient space of $L$ corresponding to the crossing.
See also Figure \ref{fig:gluing_at_crossing_truncated}.
Moreover, as depicted in Figure \ref{fig:gluing_at_semi_arc}, for each semi-arc $s$ of $D$ connecting crossings $c$ and $c^{\prime}$ (not necessarily distinct), we may glue the triangles derived from $c$ and $c^{\prime}$ along the edge $E_{s}$ corresponding to $s$.
Since the surface obtained by these gluings is exactly $\overline{S_{D}}$, the claim follows.
\begin{figure}[htbp]
 \centering
 \includegraphics[scale=0.25]{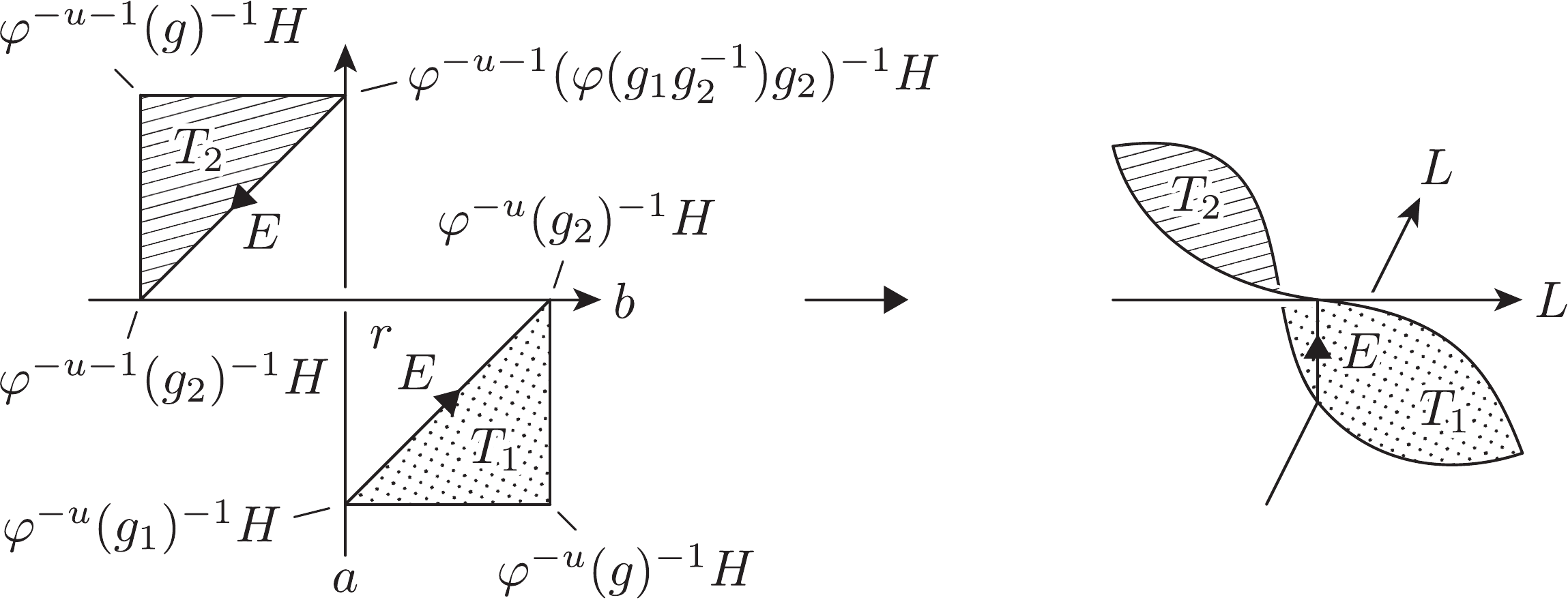}
 \caption{
  We place $T_{1}$ in $r$ and $T_{2}$ in the region opposite $r$ across the crossing, so that the vertices
  $\varphi^{-u}(g_{1})^{-1} H$, $\varphi^{-u}(g_{2})^{-1} H$, $\varphi^{- u - 1}(\varphi(g_{1} g_{2}^{-1}) g_{2})^{-1} H$, and $\varphi^{- u - 1}(g_{2})^{-1} H$
  lie on $a$, $b$, the other under arc, and $b$, respectively (left).
  Then, we glue $T_{1}$ and $T_{2}$ along $E$ in the ambient space of $L$ (right).
 }
 \label{fig:gluing_at_crossing}
\end{figure}
\begin{figure}[htbp]
 \centering
 \includegraphics[scale=0.25]{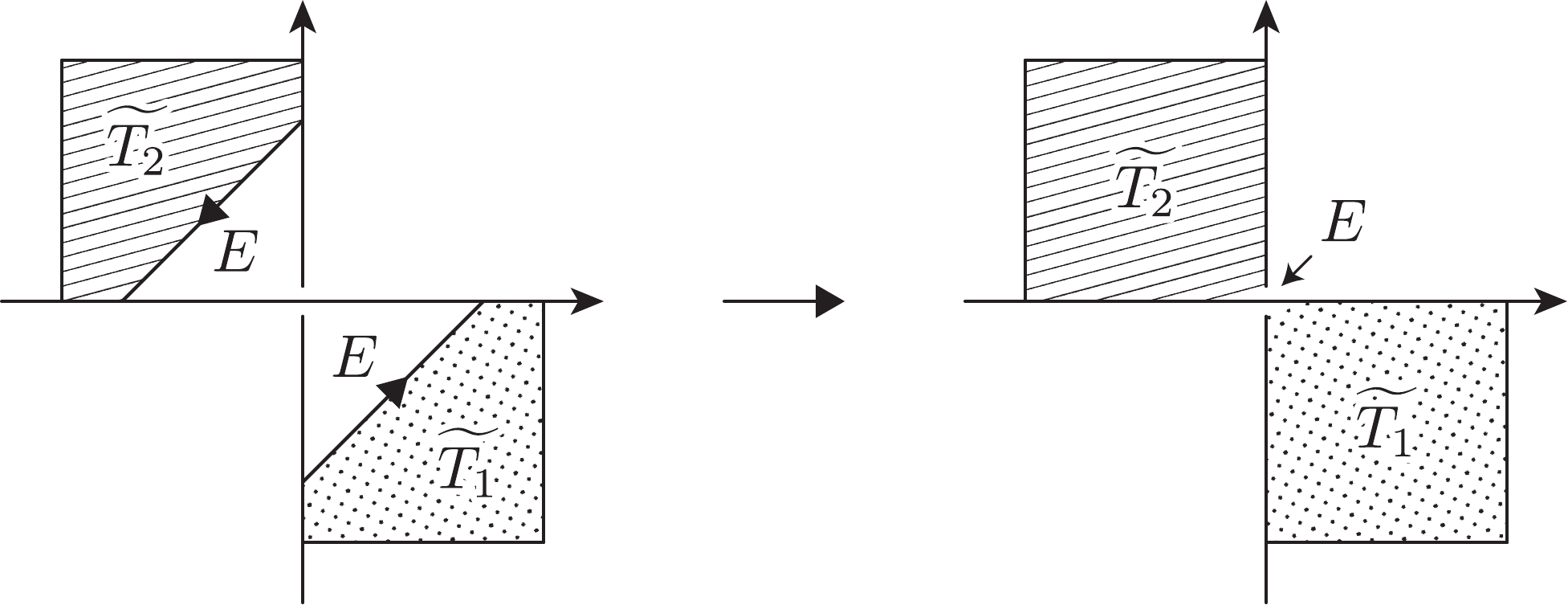}
 \caption{
  To see the relationship with the Seifert algorithm, it might be helpful to consider truncated triangles $\widetilde{T_{i}}$ (obtained from $T_{i}$ by cutting off two of its three vertices, as shown) instead of $T_{i}$.
 }
 \label{fig:gluing_at_crossing_truncated}
\end{figure}
\begin{figure}[htbp]
 \centering
 \includegraphics[scale=0.245]{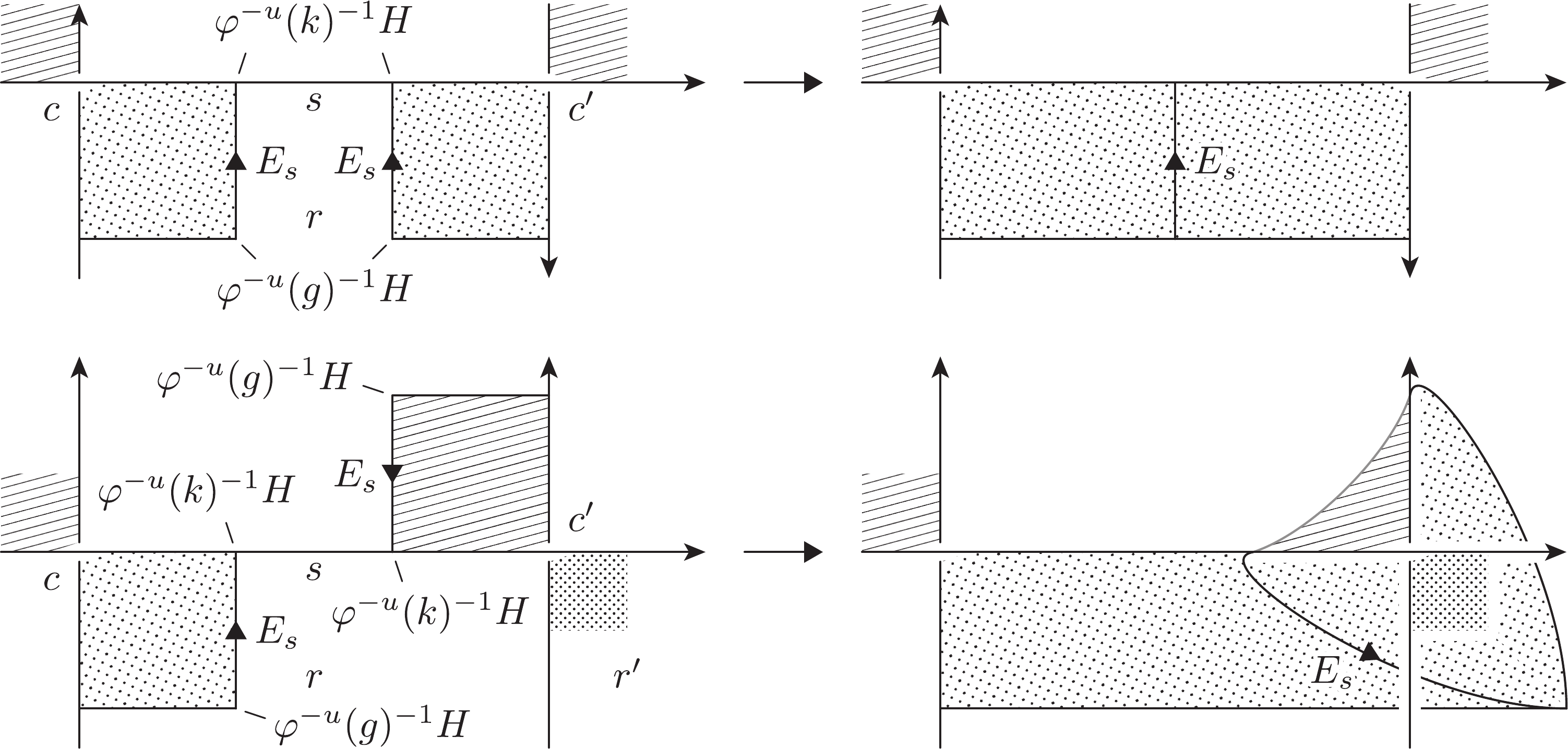}
 \caption{
  We glue the (truncated) triangles derived from $c$ and $c^{\prime}$ along $E_{s}$.
  Here, $r$ denotes the region as depicted, and we assume that $\mathscr{R}(r) = u$ and $\mathscr{A}(s) = H k$.
  We note that if $r^{\prime}$ denotes the region as depicted in the bottom figure, then $\mathscr{R}(r^{\prime}) = u - 1$.
 }
 \label{fig:gluing_at_semi_arc}
\end{figure}
\end{proof}

Let $S$ be a connected Seifert surface of $L$, which need not be canonical.
After deforming $S$ by an ambient isotopy if necessary, we may assume as depicted in the left-hand side of Figure \ref{fig:disk_band_surfaces} that $S$ admits a band projection in the sense of \cite[Section 8.B]{BZM2014}.
Suppose that $S$ has at least one band.
We may assume that each band of $S$ intersects at least one band (possibly itself) transversally in the projection.
Corresponding to each crossing of bands, turn over the upper band as depicted in the center of Figure \ref{fig:disk_band_surfaces}.
After that, we let $D_{S}$ denote the diagram of $L$ obtained from the projection.
We let $\overline{S}$ denote the closed surface obtained from $S$ by filling in its boundaries with disks.
\begin{figure}[htbp]
 \centering
 \includegraphics[scale=0.25]{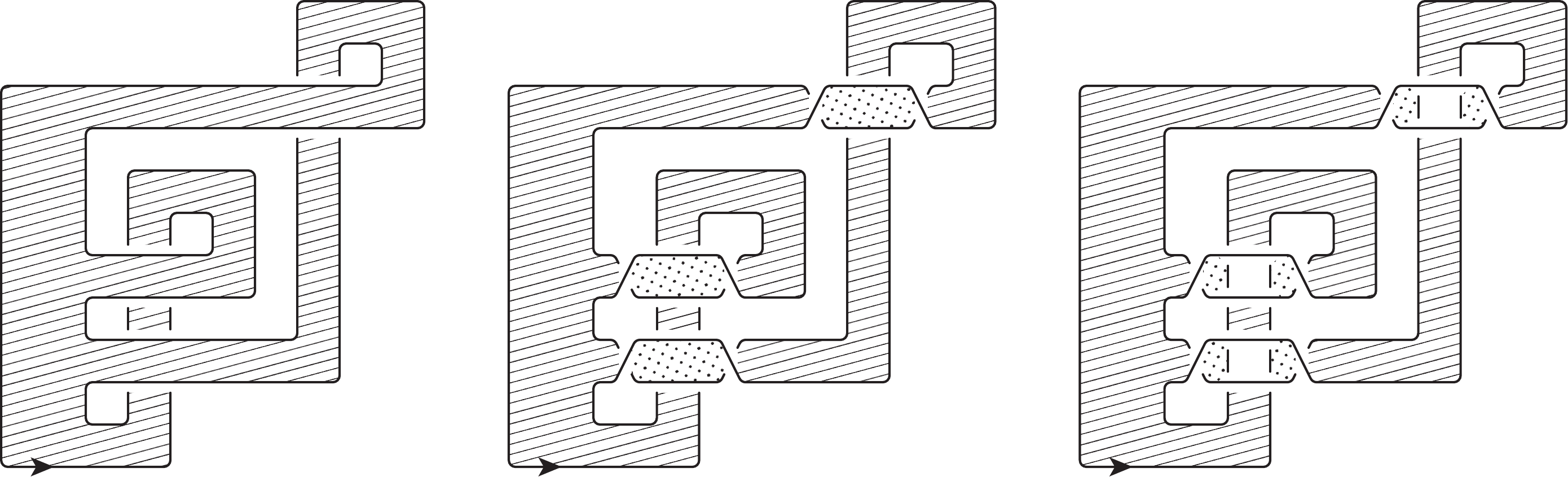}
 \caption{A band projection of $S$ (left), that after turning over the upper bands (center), and $S^{\prime}$ derived from $D_{S}$ (right).}
 \label{fig:disk_band_surfaces}
\end{figure}

\begin{theorem}
\label{thm:general_Seifert_surface}
For each $(X, \mathbb{Z})$-coloring $(\mathscr{A}, \mathscr{R})$ of $D_{S}$, the relative group 2-cycle $\psi(C(\mathscr{A}, \mathscr{R})) \in C_{2}(G / H)$ represents a triangulation of $\overline{S}$.
\end{theorem}

\begin{proof}
Let $S^{\prime}$ be a canonical Seifert surface of $L$ derived from $D_{S}$ (see the right-hand side of Figure \ref{fig:disk_band_surfaces}).
By Theorem \ref{thm:canonical_Seifert_surface}, $\psi(C(\mathscr{A}, \mathscr{R}))$ represents a triangulation of $\overline{S^{\prime}}$.
Although $\overline{S^{\prime}}$ may differ from $\overline{S}$, by regluing triangles at each crossing of bands of $S$ as illustrated in Figure \ref{fig:regluing}, we obtain a triangulation of $\overline{S}$ which $\psi(C(\mathscr{A}, \mathscr{R}))$ also represents.
\begin{figure}[htbp]
 \centering
 \includegraphics[scale=0.25]{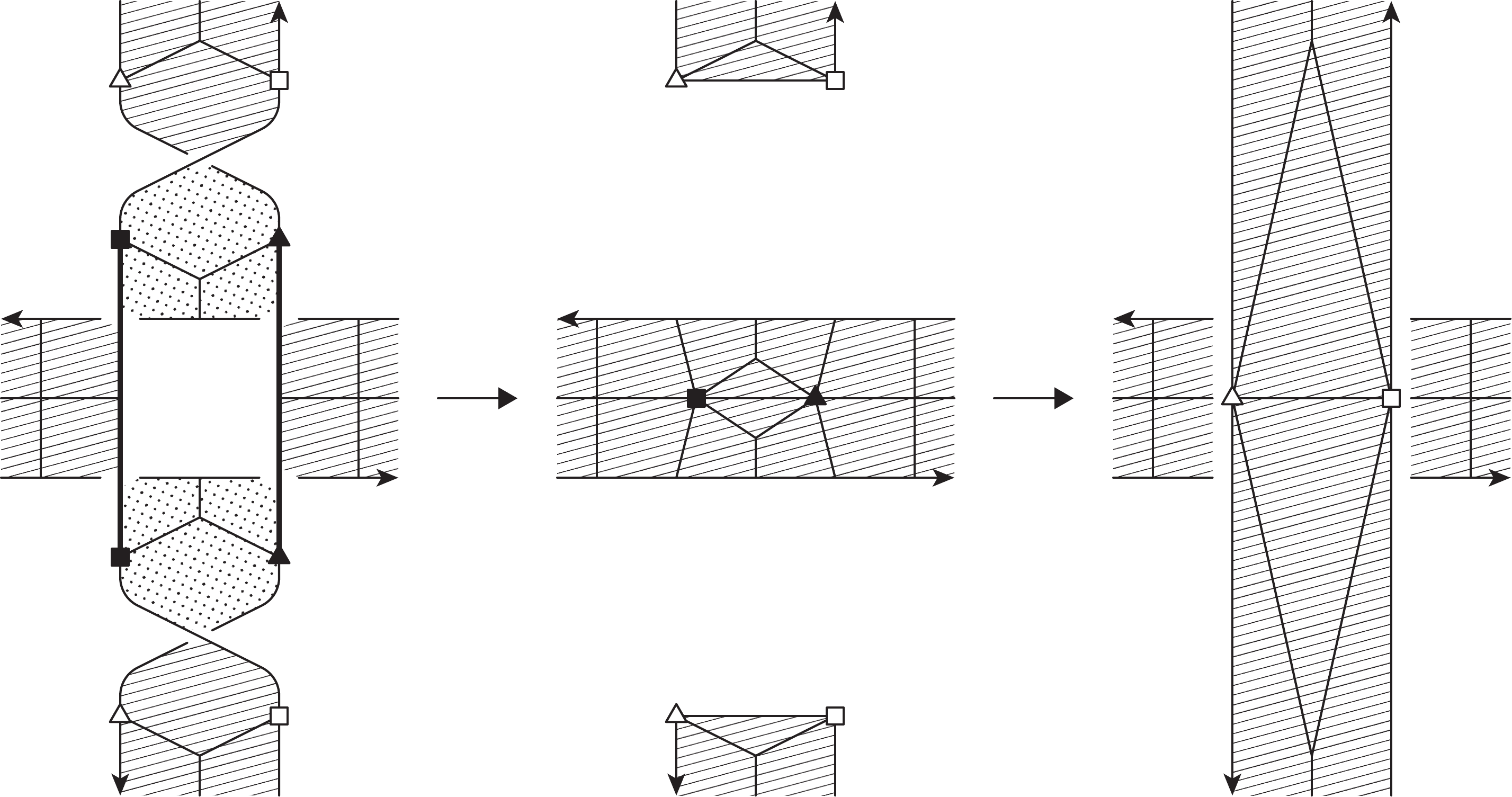}
 \caption{
  We may reglue (truncated) triangles as shown, because the vertices marked by the same symbols are labeled by the same elements of $G / H$.
  In the first step, we also stop truncating the vertices indicated by thick lines.
 }
 \label{fig:regluing}
\end{figure}
\end{proof}

\subsection{}
\label{subsec:Seifert_surface_2}

Let $L$ be an oriented 2-dimensional link, and $D$ a diagram of $L$.
Suppose that the double decker set of $L$ has no connected components being homeomorphic to circles, and all of its complementary regions are homeomorphic to disks.
We note in this situation that $D$ has at least one triple point.
Let $S_{D}$ be a Seifert hypersurface of $L$ derived from $D$ via the generalized Seifert algorithm given by Carter and Saito \cite[Section 5.4]{CS1998}.
We let $\overline{S_{D}}$ denote the topological space obtained from $S_{D}$ by coning off its boundaries.

\begin{theorem}
\label{thm:Seifert_hypersurface}
For each $(X, \mathbb{Z})$-coloring $(\mathscr{A}, \mathscr{R})$ of $D$, the relative group 3-cycle $\psi(C(\mathscr{A}, \mathscr{R})) \in C_{3}(G / H)$ represents a triangulation of $\overline{S_{D}}$.
\end{theorem}

\begin{proof}
By definition, $\psi(C(\mathscr{A}, \mathscr{R}))$ is the sum of
\begin{align*}
 & \psi(\varepsilon (\mathscr{R}(r); \mathscr{A}(a), \mathscr{A}(b), \mathscr{A}(c))) \\
 & = \psi(\varepsilon (u; H g_{1}, H g_{2}, H g_{3})) \\
 & = - \> \varepsilon (\varphi^{-u}(g)^{-1} H, \varphi^{-u}(g_{1})^{-1} H, \varphi^{-u}(g_{2})^{-1} H, \varphi^{-u}(g_{3})^{-1} H) \\
 & \phantom{=} \ + \varepsilon (\varphi^{- u - 1}(g)^{-1} H, \varphi^{- u - 1}(\varphi(g_{1} g_{2}^{-1}) g_{2})^{-1} H, \varphi^{- u - 1}(g_{2})^{-1} H, \varphi^{- u - 1}(g_{3})^{-1} H) \\
 & \phantom{=} \ + \varepsilon (\varphi^{- u - 1}(g)^{-1} H, \varphi^{- u - 1}(\varphi(g_{1} g_{3}^{-1}) g_{3})^{-1} H, \\
 & \hspace{12em} \varphi^{- u - 1}(\varphi(g_{2} g_{3}^{-1}) g_{3})^{-1} H, \varphi^{- u - 1}(g_{3})^{-1} H) \\
 & \phantom{=} \ - \varepsilon (\varphi^{- u - 2}(g)^{-1} H, \varphi^{- u - 2}(\varphi^{2}(g_{1} g_{2}^{-1}) \varphi(g_{2} g_{3}^{-1}) g_{3})^{-1} H, \\
 & \hspace{12em} \varphi^{- u - 2}(\varphi(g_{2} g_{3}^{-1}) g_{3})^{-1} H, \varphi^{- u - 2}(g_{3})^{-1} H)
\end{align*}
over all triple points of $D$, where $\varepsilon$, $r$, $a$, $b$ and $c$ are as depicted in Figure \ref{fig:around_crossing_point_and_triple_point}.
We may glue the above four tetrahedra as depicted in Figure \ref{fig:gluing_at_triple_point}.
See also Figure \ref{fig:gluing_at_triple_point_truncated}.
In the generalized Seifert algorithm, each triple point of $D$ is smoothed as depicted in Figure \ref{fig:smoothing_at_triple_point}.
Therefore, this gluing yields a triangulation of a part of $\overline{S_{D}}$ corresponding to the triple point.
Moreover, in a similar manner illustrated in Figure \ref{fig:gluing_at_semi_arc}, for each double point curve connecting two triple points (not necessarily distinct), we may glue the tetrahedra derived from the triple points along two faces corresponding to the double point curve.
By the assumption on the double decker set, these gluings yield a triangulation of $\overline{S_{D}}$.
\begin{figure}[htbp]
 \centering
 \includegraphics[scale=0.25]{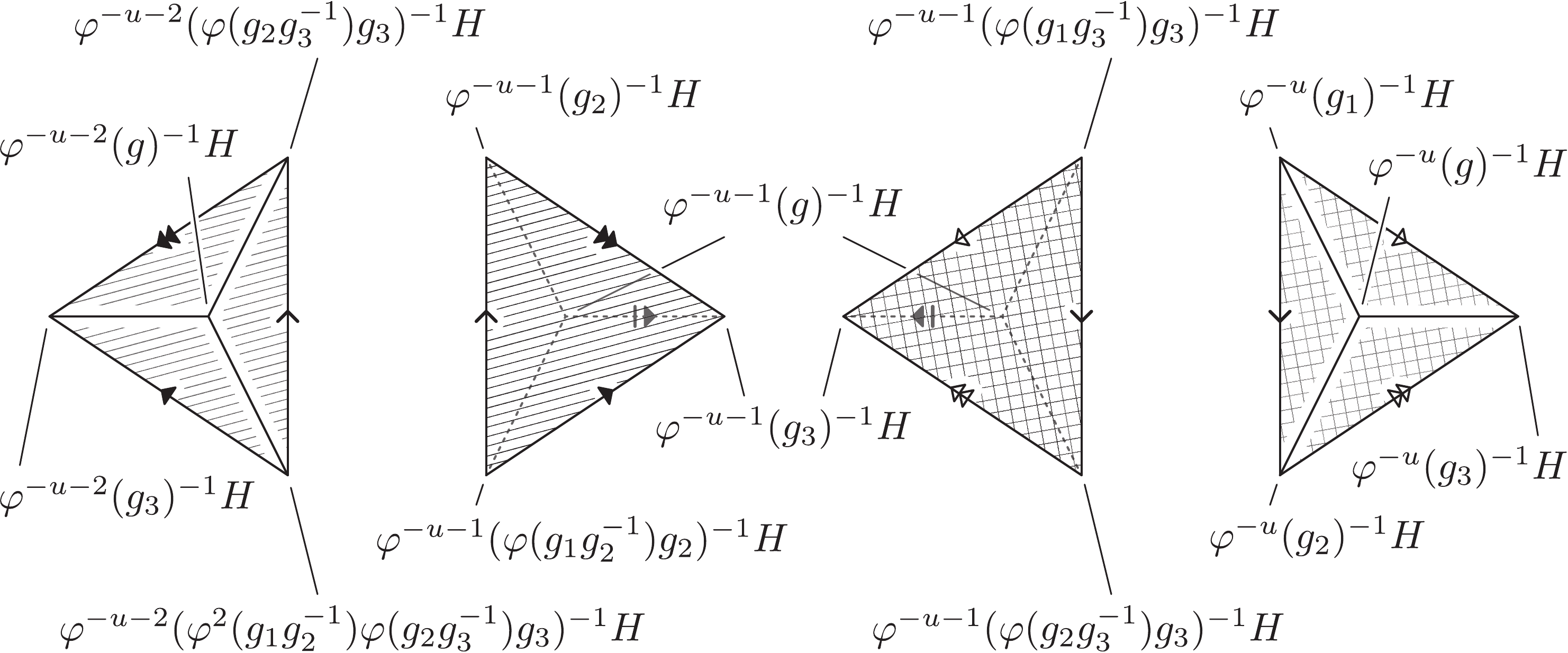}
 \caption{
  We glue the faces indicated by the same patterns so that the arrows assigned to the boundary edges match.
  Moreover, we also glue the edges marked by \includegraphics[scale=0.25]{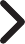} or \includegraphics[scale=0.25]{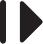} so that their orientations coincide.
  We note that these gluings can be performed in $S^{4}$.
 }
 \label{fig:gluing_at_triple_point}
\end{figure}
\begin{figure}[htbp]
 \centering
 \includegraphics[scale=0.25]{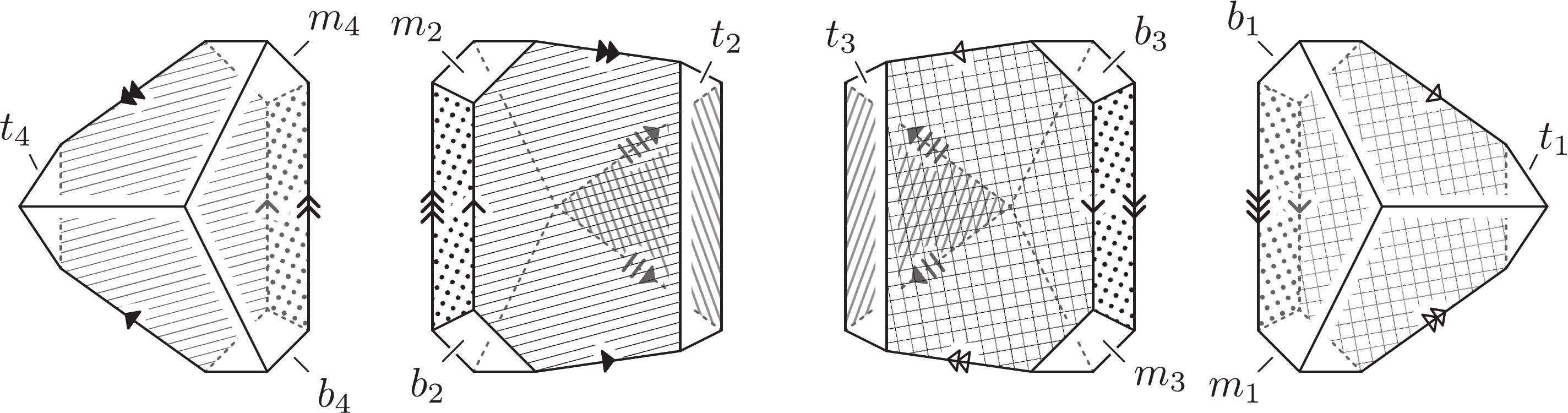}
 \caption{
  To verify the relationship with the generalized Seifert algorithm, it might be helpful to cut off some vertices and edges of the four tetrahedra as shown.
  Here, faces $t_{i}$, $m_{i}$ and $b_{i}$ ($1 \leq i \leq 4$) correspond to the parts of the top, middle and bottom sheets of the triple point (indicated in Figure \ref{fig:smoothing_at_triple_point}), respectively.
 }
 \label{fig:gluing_at_triple_point_truncated}
\end{figure}
\begin{figure}[htbp]
 \centering
 \includegraphics[scale=0.25]{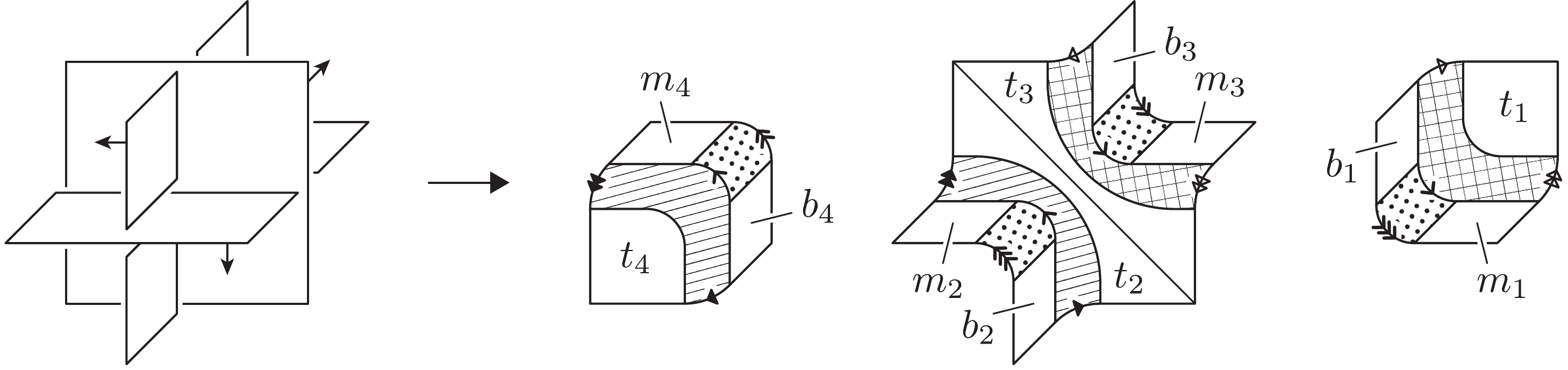}
 \caption{
  By smoothing, each triple point of $D$ is decomposed into three pieces as depicted.
  We note that, in the generalized Seifert algorithm, the hatched and checked areas are subsequently glued (via 1-handles) so that the arrows assigned to boundary edges match.
  After that, the dotted areas are glued in the same manner.
 }
 \label{fig:smoothing_at_triple_point}
\end{figure}
\end{proof}

\section*{Acknowledgments}
The author was supported by JSPS KAKENHI Grant Number JP25K07014.
He thanks Professor Masahico Saito for pointing out a historical inaccuracy in an earlier version of this paper.

\bibliographystyle{amsplain}

\end{document}